\documentclass{article}
\usepackage{amsthm,amsfonts, amsbsy, amssymb,amsmath,graphicx}
\usepackage{graphics}

\newtheorem{thm}{Theorem}
\newtheorem{lm}{Lemma}
\newtheorem{re}{Remark}
\newtheorem{crl}{Corollary}

\newtheorem{problem}{Problem}
\newtheorem{prop}{Proposition}
\newtheorem{st}{Statement}

\newcounter{tdfn}
\newenvironment{dfn}
{\vspace{0.15cm}{\bf Definition \arabic{tdfn}.}} {\par
\addtocounter{tdfn}{1}}

\author{Vassily Olegovich Manturov}

\date{}

\title{Embeddings of four-valent framed graphs into $2$-surfaces}

\begin{document}

\maketitle

\begin{abstract}

It is well known that the problem of defining the least (highest)
genus where a given graph can be embedded is closely connected to
the problem of embedding special {\em four-valent framed graphs},
i.e. 4-valent graphs with opposite edge structure at vertices
specified. This problem has been studied, and some cases (e.g.,
recognizing planarity) are known to have a polynomial solution.

The aim of the following paper is to connect the problem above to
several problems which arise in knot theory and combinatorics:
Vassiliev invariants and weight systems coming from Lie algebras,
Boolean matrices etc., and to give both partial solutions to the
problem above and new formulations of it in the language of knot
theory.

\end{abstract}

\newcommand{\skcrossr}{\raisebox{-0.25\height}{\includegraphics[width=0.5cm]{skcrossr.eps}}}
\newcommand{\skcrossl}{\raisebox{-0.25\height}{\includegraphics[width=0.5cm]{skcrossl.eps}}}
\newcommand{\skkinkr}{\raisebox{-0.25\height}{\includegraphics[width=0.5cm]{skkinkr.eps}}}
\newcommand{\skkinkl}{\raisebox{-0.25\height}{\includegraphics[width=0.5cm]{skkinkl.eps}}}
\newcommand{\skroh}{\raisebox{-0.25\height}{\includegraphics[width=0.5cm]{skroh.eps}}}
\newcommand{\skrov}{\raisebox{-0.25\height}{\includegraphics[width=0.5cm]{skrov.eps}}}
\newcommand{\skrtwhh}{\raisebox{-0.25\height}{\includegraphics[width=0.5cm]{skrtwhh.eps}}}
\newcommand{\skrtwvh}{\raisebox{-0.25\height}{\includegraphics[width=0.5cm]{skrtwvh.eps}}}
\newcommand{\skrtwhv}{\raisebox{-0.25\height}{\includegraphics[width=0.5cm]{skrtwhv.eps}}}
\newcommand{\skrtwvv}{\raisebox{-0.25\height}{\includegraphics[width=0.5cm]{skrtwvv.eps}}}
\newcommand{\skcrv}{\raisebox{-0.25\height}{\includegraphics[width=0.5cm]{skcrv.eps}}}
\newcommand{\skcrh}{\raisebox{-0.25\height}{\includegraphics[width=0.5cm]{skcrh.eps}}}
\newcommand{\vcross}{\raisebox{-0.25\height}{\includegraphics[width=0.5cm]{vcross.eps}}}
\newcommand{\skcurl}{\raisebox{-0.25\height}{\includegraphics[width=0.5cm]{skcurl.eps}}}
\newcommand{\skkinkv}{\raisebox{-0.25\height}{\includegraphics[width=0.5cm]{skkinkv.eps}}}

\section{Introduction}

We address the following question. Given a 4-valent graph $\Gamma$
with each vertex endowed with opposite half-edge structure.

A natural question is to study the highest (least) genus of the
surface the graph can be embedded into. Of course, we mean only
embeddings splitting the surface into $2$-cells. We shall address
this general question later in this paper. First, we shall consider
the following partial cases of it. One of them, more general, deals
with embedded graphs whose first ${\bf Z}_{2}$-homology class is
orienting. As a partial case of this, we address the following

\begin{problem}
Which is the least possible (highest possible) genus of a
$2$-surface $S$ this graph can be embedded into in such a way that
the embedding represents the zero homology class in the surface
(alternatively, the complement to the graph is checkerboard
colourable).\label{prob1}
\end{problem}

Embeddings of such graphs representing the ${\bf Z}_{2}$-homology
class are well studied for the case of the plane (see e.g.,
\cite{LovaszMarx,ReadRosenstiel,VassConj}) and in the general case
(see e.g., \cite{LinsRichterSchank,CrapoRosenstiel}).

In fact, any embedding of a $4$-graph in ${\bf R}^{2}$ defines a
checkerboard colouring on the set of  faces because the plane has
trivial first homology. On the other hand, any graph $\Gamma$
embedded into a $2$-surface $S$ (orientable or not) can be
transformed into a $4$-graph by taking the medial graph $\Gamma'$:
the vertices of $\Gamma'$ are the middle points of the edges of
$\Gamma$, the edges of $\Gamma'$ connect {\em adjacent} edges
(sharing the same angle), and faces correspond to faces (white) and
vertices (black) of $\Gamma$, see Fig. \ref{graphs}.

\begin{figure}
\centering\includegraphics[width=300pt]{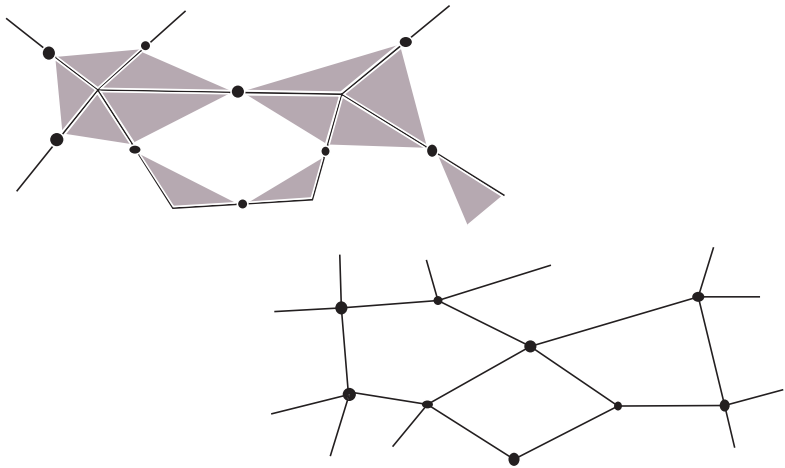} \caption{Any
embedded graph generates a $4$-valent framed graph} \label{graphs}
\end{figure}

Such $4$-valent graphs appeared with many names in different
problems of low-dimensional topology: as {\em atoms} (see rigorous
definition ahead), originally due to Fomenko \cite{F}, see the
connection between atoms and knots in \cite{MyBook}, they are
connected to Grothendieck's {\em dessins d'enfant}, see
\cite{LandoZvonkine} and \cite{LinYiFeng}.

There is a nice connection between combinatorics of Vassiliev
invariants and other invariants of knots and virtual knots and many
well-known functions on graphs, see \cite{CDBook} and references
therein.

Finally, the genus of the atom (the genus of the checkerboard
surface we are interested in) is closely connected to the estimates
of the thickness for Khovanov and Ozsv\'ath-Szab\'o homology for
classical and virtual knots, see \cite{Minimal} and \cite{Lowrance}.

In \cite{CrapoRosenstiel} there was a reformulation of the problem
stated above in terms of ranks of some matrices.

We give another (very easy) formulation in terms of ranks of
matrices which is closely connected to knot theory.

\begin{problem}
Given a symmetric ${\bf Z}_{2}$-matrix $M$ of size $n\times n$, find
a splitting of the set of indices $\{1,\dots, n\}$ into $I\sqcup J$
such that for the corresponding square matrices $M_{I}$ and $M_{J}$,
the sum of ranks $rk(M_{I})+rk(M_{J})$ is minimal (resp.,
maximal).\label{prob2}
\end{problem}

This seems to be the easiest reformulation of the initial problem.
Of course, we are looking for a solution which would either be fast
(say, having polynomial time in the number of vertices) or connected
to some interesting mathematical problems.

In knot theory, the study of classical knots is closely connected to
the so-called {\em $d$-digarams}, chord diagrams with $2$ sets of
pairwise unlinked chords (see rigorous definition ahead). It turns
out that these diagrams play a special role in the chord diagram
algebra having the highest possible degree of the Vassiliev
invariants coming from $sl(n)$ (see \cite{MitRutwig}). On the other
hand, these are precisely those diagrams corresponding (in sense of
\cite{MyBook}) to planar $4$-valent graphs.

This is not incidental. In fact, the generating function for such
embeddings is closely connected to the $sl(n)$-weight system, and
latter sometimes gives estimates for the genus of the atom where the
framed graph can be embedded into.

The paper is organized as follows. In the next section, we give the
definitions of atom, chord diagram, $d$-diagram, virtual link and
establish a connection between them and embedded graphs. We also
give a proof of a conjecture due to Vassiliev (stated in \cite{Vas})
and proved in \cite{VassConj} saying that the only obstruction to
the planar embeddability of such graphs is the existence of two
cycles with no common edges with exactly one {\em transverse}
intersection point.

Later, we also give criteria for embeddability of framed graphs to
the real projective space and in the Klein bottle.

In section 3, we define the Kauffman bracket for the virtual knots
and we recall a result by Soboleva \cite{Soboleva} about the number
of circles which appear after a surgery along a chord diagram. This
will lead us to the reformulation of Problem \ref{prob1} as Problem
\ref{prob2}.

Section 4 will be devoted to the connection between chord diagrams,
weight systems and Vassiliev's invariants coming from Lie algebras
in sense of Bar-Natan \cite{BN}. We shall prove a theorem giving an
estimate in terms of $sl(n)$-invariants.

The last section will be devoted to the discussion and open
problems.

\subsection{Atoms and Knots}

A four-valent planar graph $\Gamma$ generates a natural checkerboard
colouring of the plane by two colours (adjacent components of the
complement ${\bf R}^{2}\backslash \Gamma$ have different colours).

This construction perfectly describes the role played by {\em
alternating diagrams} of classical knots. Recall that a link diagram
is {\em alternating} if while walking along any component one
alternates over- and underpasses. Another definition of an
alternating link diagram sounds as follows: fix a checkerboard
colouring of the plane (one of the two possible colourings). Then,
for every vertex the colour of the region corresponding to the angle
swept by going from the overpass to the underpass in the
counterclockwise direction is the same.

Thus, planar graphs with natural colourings somehow correspond to
alternating diagrams of knots and links on the plane: starting with
a graph and a colouring, we may fix the rule for making crossings:
if two edges share a black angle, then the we decree the left one
(with respect to the clockwise direction) to form an overcrossing,
and the right one to be an undercrossing, see Fig. \ref{overunder}.
Thus, colouring a couple of two opposite angles corresponds to a
choice of a pair of opposite edges to form an overcrossing and vice
versa.

\begin{figure}
\centering\includegraphics[width=200pt]{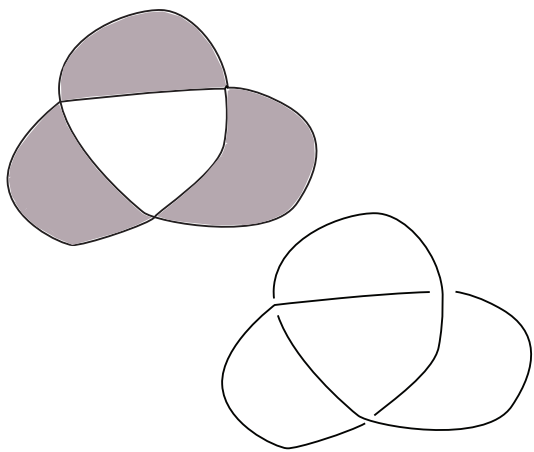}
\caption{Surface colouring and a knot diagram} \label{overunder}
\end{figure}

Now, if we take an arbitrary link diagram and try to fix the
colouring of regions by colouring angles according to the rule
described above, we see that generally it is impossible unless the
initial diagram is alternating: we can just get a region on the
plane where colourings at two adjacent angles disagree. So,
alternating diagrams perfectly match colourings of the $2$-sphere
(think of $S^{2}$ as a one-point compactification of ${\bf R}^{2}$).
For an arbitrary link, we may try to take colours and attach cells
to them in a way that the colours would agree, namely, the circuits
for attaching two-cells are chosen to be rotating circuits, where we
always turn inside the angle of one colour.

This leads to the notion of {\em atom}. An {\em atom} is a pair
$(M,\Gamma)$ consisting of a $2$-manifold $M$ and a graph $\Gamma$
embedded $M$ together with a colouring of $M\backslash \Gamma$ in a
checkerboard manner. Here $\Gamma$ is called the {\em frame} of the
atom, whence by {\em genus} (resp., {\em Euler characteristic}) of
the atom we mean that of the surface $M$.

Note that the atom genus is also called the {\em Turaev genus},
\cite{TuraevGenus}.

Certainly, such a colouring exists if and only if $\Gamma$
represents the trivial ${\bf Z}_{2}$-homology class in $M$.

Thus, gluing cells to some rotating circuits on the diagram, we get
an atom, where the shadow of the knot plays the role of the frame.
Note that the structure of opposite half-edges on the plane
coincides with that on the surface of the atom.

Now, we see that atoms on the sphere are precisely those
corresponding to alternating link diagrams, whence non-alternating
link diagrams lead to atoms on surfaces of a higher genus.

In some sense, the genus of the atom is a measure of how far a link
diagram is from an alternating one, which leads to generalisations
of the celebrated Kauffman-Murasugi theorem, see \cite{MyBook} and
to some estimates concerning the Khovanov homology \cite{Minimal}.

Having an atom, we may try to embed its frame in ${\bf R}^{2}$ in
such a way that the structure of opposite half-edges at vertices is
preserved. Then we can take the ``black angle'' structure of the
atom to restore the crossings on the plane.

In \cite{AtomsandKnots} it is proved that the link isotopy type does
not depend on the particular choice of embedding of the frame into
${\bf R}^{2}$ with the structure of opposite edges preserved. The
reason is that such embeddings are quite rigid.

The atoms whose frame is embeddable in the plane with opposite
half-edge structure preserved are called {\em height} or vertical.

However, not all atoms can be obtained from some classical knots.
Some abstract atoms may be quite complicated for its frame to be
embeddable into ${\bf R}^{2}$ with the opposite half-edges structure
preserved. However, if it is impossible to {\em immerse} a graph in
${\bf R}^{2}$, we may embed it by marking artifacts of the embedding
(we assume the embedding to be generic) by small circles.

This leads to a connection between atoms and {\em virtual knots}
which perfectly agrees with {\em virtual knot theory} proposed by
Kauffman in \cite{KaV}.

\begin{dfn}
A {\em virtual diagram} is a $4$-valent diagram in ${\bf R}^2$ where
each crossing is either endowed with a classical crossing structure
(with a choice for underpass and overpass specified) or just said to
be virtual and marked by a circle.
\end{dfn}

\begin{dfn}
A {\em virtual link} is an equivalence class of virtual link diagram
modulo generalized Reidemeister moves. The latter consist of usual
Reidemeister moves referring to classical crossings and the {\em
detour move} that replaces one arc containing only virtual
intersections and self-intersection by another arc of such sort in
any other place of the plane, see Fig. \ref{detour}.
\end{dfn}

\begin{figure}
\centering\includegraphics[width=200pt]{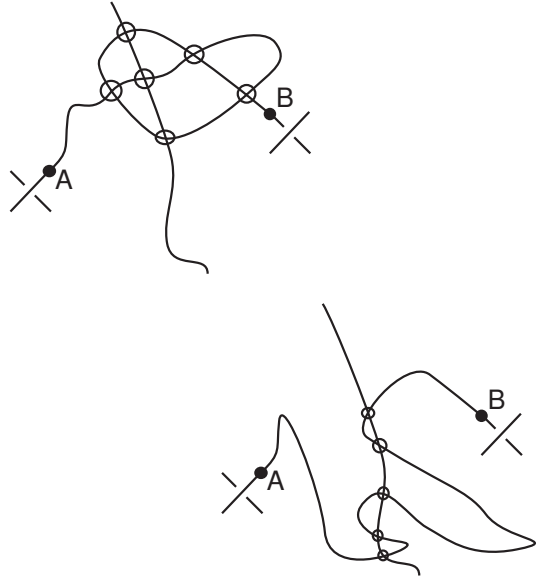} \caption{The
detour move} \label{detour}
\end{figure}

Having freedom of immersing knot diagrams into ${\bf R}^{2}$ instead
of just embedding, we are able to make virtual diagrams out of
atoms. Obviously, since we disregard virtual crossings, the most we
can expect is the well-definiteness up to detours. However, this
allows us to get different virtual link types from the same atom,
since for every vertex $V$ of the atom with four emanating
half-edges $a,b,c,d$ (ordered cyclically on the atom) we may get two
different clockwise-orderings on the plane of embedding, $(a,b,c,d)$
and $(a,d,c,b)$. This leads to a move called {\em virtualisation}.

\begin{dfn}
By a {\em virtualisation} of a classical crossing we mean a local
transformation shown in Fig. \ref{virtualisation}.
\end{dfn}

\begin{figure}
\centering\includegraphics[width=200pt]{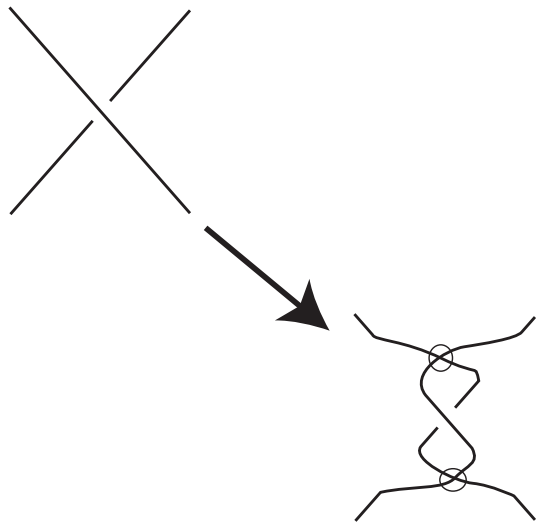}
\caption{Virtualisation} \label{virtualisation}
\end{figure}

The above statements summarise as
\begin{prop}(see., e.g. \cite{MyBook}).
Let $L_1$ and $L_2$ be two virtual links obtained from the same atom
by using different immersions of its frame. Then $L_1$ differs from
$L_2$ by a sequence of (detours and) virtualisations.
\end{prop}

Obviously, the inverse operation of obtaining an atom from a virtual
diagram is well defined.

Note that many famous invariants of classical and virtual knots
(Kauffman bracket, Khovanov homology, Khovanov-Rozansky homology
\cite{KhR1,KhR2}) do not change under the virtualisation, which
supports the {\em virtualisation conjecture}: if for two classical
links $L$ and $L'$ there is a sequence $L=L_{0}\to \cdots\to L_{n}$
of virtualisations and generalised Reidemeister moves then $L$ and
$L'$ are classically equivalent (isotopic).

Note that the usual virtual equivalence implies classical
equivalence for classical knots, see \cite{GPV}.

\begin{dfn}
By a chord diagram we mean a cubic graph with a selected oriented
cycle $S^{1}$ so that the remaining edges represent a collection of
disjoint {\em chords}.

A chord diagram is called framed if every chord is marked by either
$+1$ or $-1$. A chord diagram without framings is assumed to have
all chords positive.

Two chords $A,B$ of a chord diagram are {\em linked} if the ends
$A_{1}$ and $A_{2}$ of the first chord lie in different components
of $S^{1}\backslash \{B_{1},B_{2}\}$.

Analogously, one defines a chord diagram on $m$ circles; there
should be $m$ oriented cycles with some points connected by chords.

The {\em intersection graph} of a chord diagram is a graph whose
vertices are in one-to-one correspondence with the edges of chord
diagram and the vertices of the graph. Two vertices are connected by
an edge iff the corresponding chords are linked.

A chord diagram (with all chords framed positively) is a {\em
$d$-diagram} if the corresponding intersection graph is bipartite.

If the chord diagram is framed then the corresponding graph obtains
framings at vertices.
\end{dfn}

Now, assume we have an atom with exactly one white cell. Then the
whole information about the atom can be obtained from a rotating
circuit along the boundary of this cell. Namely, consider a  walk
along this boundary (in any direction) as a map from $S^{1}$ to the
atom, and connect the preimages of vertices of the atom by chords.
We thus construct a {\em framed chord diagram}, where the framing is
positive whence the orientations of the two two circles locally
agree or negative otherwise, see Fig. \ref{twosortsoforientations}.

\begin{figure}
\centering\includegraphics[width=200pt]{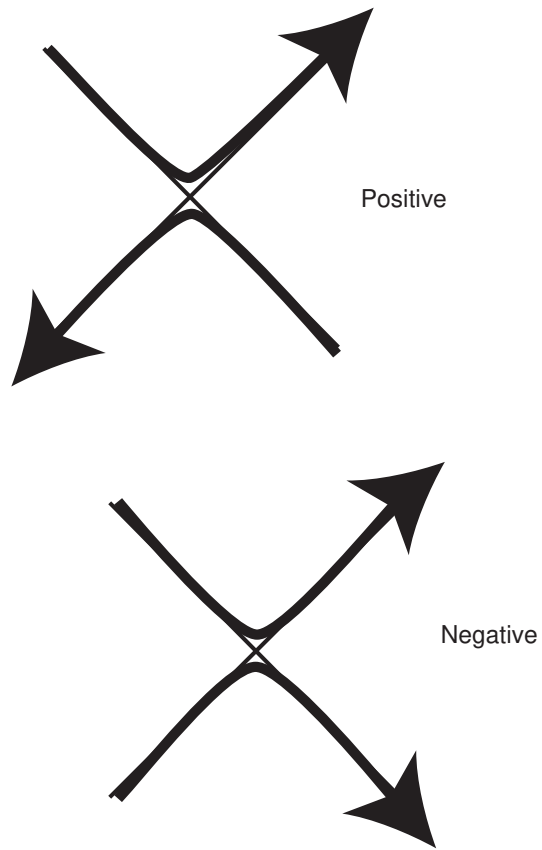} \caption{Two
types of orientations in a neighbourhood of a vertex}
\label{twosortsoforientations}
\end{figure}

From this chord diagram, one easily restores the atom and, thus, the
corresponding virtual link up to detours and virtualisations.

{\bf Notation:} For a chord diagram $C$ with one cell denote the
corresponding atom by $A(C)$ and denote the corresponding virtual
knot (considered up to detours and virtualisations) by $K(C)$.

\subsection{Acknowledgements}

I explain my gratitude for useful discussions to V.A.Vassiliev,
A.T.Fomenko, L.H.Kauffman, S.K.Lando, S.V.Duzhin, S.V.Chmutov,
O.M.Kasim-Zade, O.R. Campoamor-Stursberg.

\section{Chord Diagrams, $1$-Dimensional Surgery and The Kauffman
Bracket}

The Kauffman bracket \cite{KauffmanBracket} is a very useful model
for understanding the Jones polynomial \cite{Jones}. The Kauffman
bracket associates with a virtual knot diagram a Laurent polynomial
in one variable $a$ associated to every virtual diagram. After a
small normalisation (multiplciation by a power of $(-a)$) it gives
an invariant for virtual links.

This invariant can be read out of atom corresponding to a knot
diagram. Namely, take an atom $V$ with $n$ vertices corresponding to
a virtual diagram $L$ with $n$ classical crossings and call a {\em
state} a choice of the couple of black or white angles at every
vertex of $V$. Every such choice gives rise to a collection of
closed curves on $V$ whose boundaries contain all the edges of $V$,
see Fig. \ref{astateonanatom}, and at each crossing the curves turn
locally from one edge to an adjacent edge sharing the same angle of
the prefixed colour.

\begin{figure}
\centering\includegraphics[width=200pt]{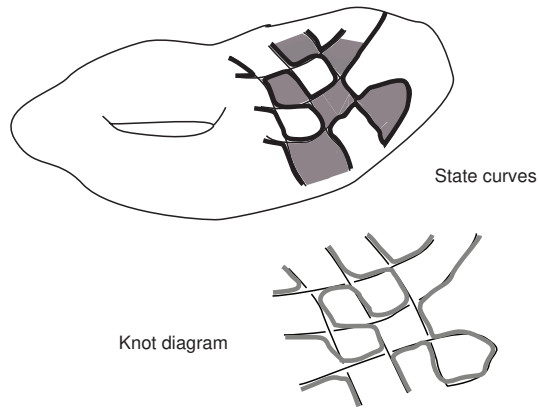} \caption{State
curves drawn on an atom} \label{astateonanatom}
\end{figure}

Thus, having $2^{n}$ states of the atom, we define the Kauffman
bracket of it as

\begin{equation}
\langle V\rangle
=\sum_{s}a^{\alpha(s)-\beta(s)}(-a^{2}-a^{-2})^{\gamma(s)-1},\label{Kauffmanbracket}
\end{equation}
where the sum is taken over all states $s$ of the diagram,
$\alpha(s)$ and $\beta(s)$ denote the number of white and black
angles in the state (thus, $\alpha(s)+\beta(s)=n$ and $\gamma(s)$
denotes the number of curves in the state).

As it was said, the Kauffman bracket is invariant under the
virtualisation. Thus, it is not surprising that it can be read out
of the corresponding atom.

If the atom $A$ is obtained from a (framed) chord diagram $C$, then
one can construct the Kauffman bracket $\langle C(A)\rangle$.

Thus, one obtains a function $f$ on framed chord diagram valued in
Laurent polynomials in $a$. We shall return to that function because
it is connected to the Vassiliev invariants of knots and
$J$-invariants of closed curves (Lando, \cite{Lando}).

Assume now we have a framed graph (a graph with each vertex labeled
either positively or negatively).

By a {\em rotating circuit} of a framed graph $\Gamma$ we mean a map
from the oriented circle $S^{1}\to \Gamma$ which is homeomorphic
outside preimages of vertices of $\Gamma$, and each vertex of
$\Gamma$ has precisely two preimages such that the corresponding
neighbourhoods of them on the circle switch from one edge to an edge
not opposite to it.

Every rotating circuit gives rise to a framed chord diagram
associated with a graph: this chord diagram consists of the circle
$S^{1}$ and chords connecting those points of $S^{1}$ having the
same image in $\Gamma$. A chord is {\em positive} if for the
corresponding vertex the two emanating edges are opposite.

Then it may or may not be represented as an intersection graph of a
chord diagram (see \cite{Bouchet} for the details) for which it is
an intersection graph. Moreover, if such a chord diagram exists, it
should not be unique, see, e.g., Fig.~\ref{nonuniqueness}.

\begin{figure}
\centering\includegraphics[width=200pt]{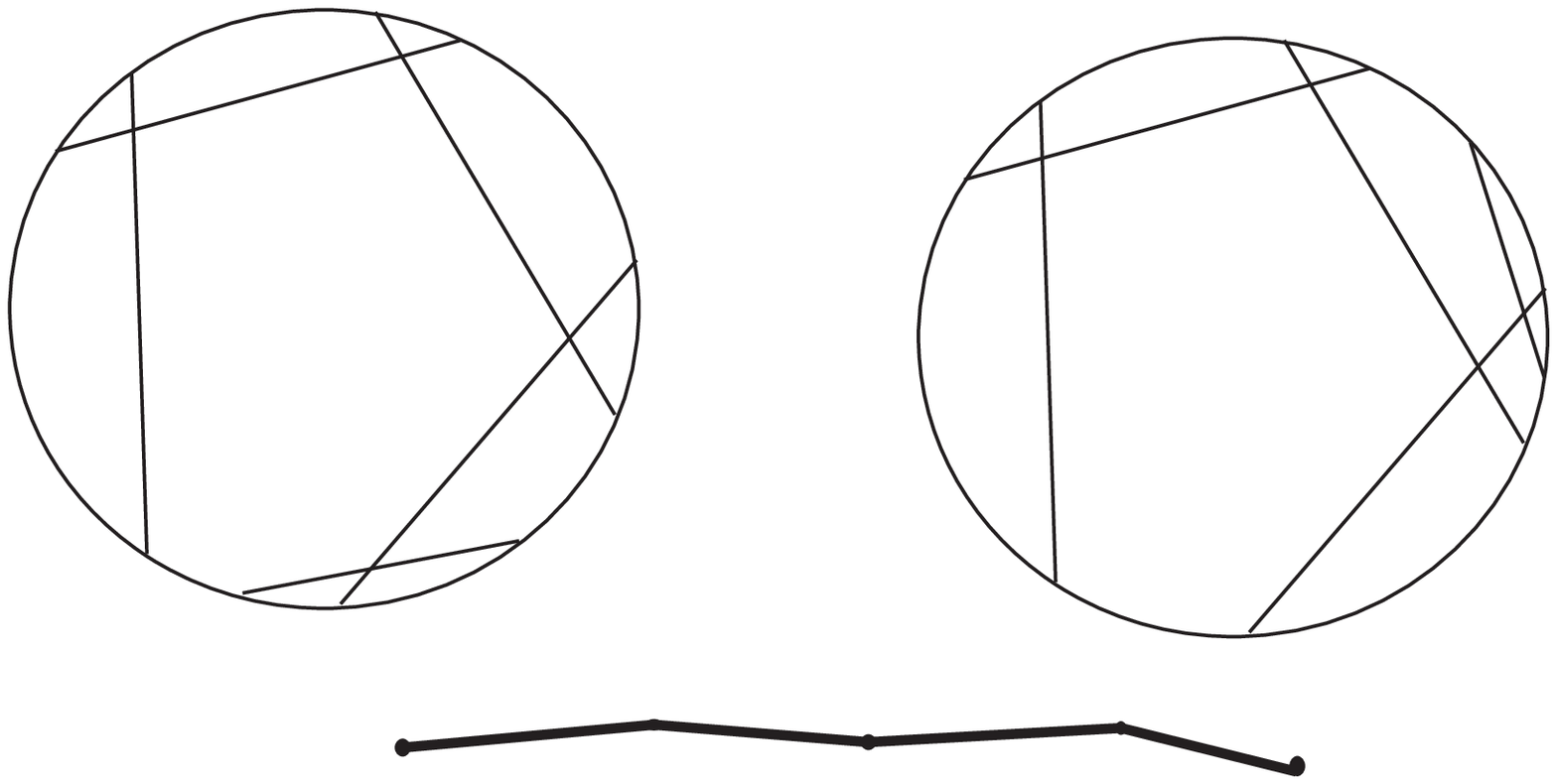} \caption{Two
chord diagrams with the same intersection graph}
\label{nonuniqueness}
\end{figure}

This non-uniqueness usually corresponds to so called {\em mutations}
of virtual knots.

The mutation operation (shown in the top of Fig. \ref{muta}) cuts a
piece of a knot diagram inside a box turn is by a half-twist and
returns to the initial position.

It turns out that the mutation operation is expressed in terms of
chord diagrams in almost the same way: one cuts a piece of diagram
with $4$ ends and exchanges the top and the bottom part of it (see
bottom picture of Fig.~\ref{muta}). Exactly this operation
corresponds to the mutation from both Gauss diagram and rotating
circuit points of view.

\begin{figure}
\centering\includegraphics[width=200pt]{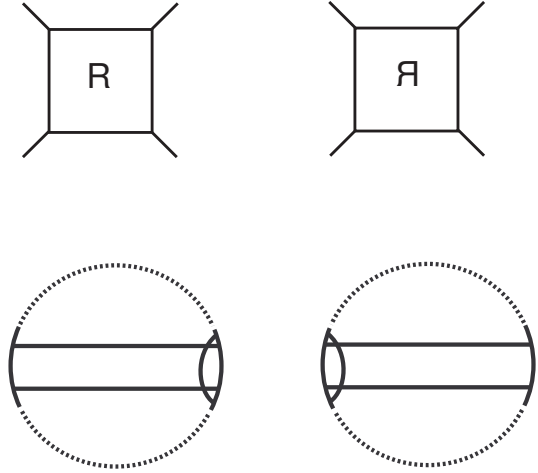} \caption{Mutation
and its chord diagram presentation} \label{muta}
\end{figure}

In the bottom part of Fig. \ref{muta} chords whose end points belong
to the ``dotted'' area remain the same; the other chords are
reflected as a whole.

Regarded from the point of view of Gauss diagrams and the Vassiliev
knot invariants, this non-uniqueness corresponds to {\em mutations}
of classical links as well. Namely, S.K.Lando and S.V.Chmutov
\cite{ChmutovLando} proved the following

\begin{thm}
A Vassiliev invariant does not detect mutant knots if and only if
the corresponding weight system depends only on the intersection
graph of the chord diagram.
\end{thm}

It is well-known that the Kauffman bracket does not detect
mutations. Thus, one might guess that the corresponding Kauffman
bracket can be read out of the intersection graph.

Surprisingly, the Kauffman bracket can be defined in a meaningful
way even for those framed graphs which can not be represented as
intersection graphs of chord diagrams.

Having a chord diagram, we can treat the states of the corresponding
Kauffman bracket as collections of chords: we set the initial state
$s_0$ to be empty collection of curves (with
$\alpha(s)=n,\beta(s)=0$), and with each state $s$ we associate a
collection of chords correpsonding to those vertices of the atom
where $s$ differs from $s_0$.

Now, if we are able to calculate {\em how many circles we have in
each state}, we can apply  (\ref{Kauffmanbracket}) to calculate the
Kauffman bracket.

This can be seen from a chord diagram after introducing the notion
of {\em surgery along a chord}.

Given a chord diagram $D$ on $n$ circles $C_{1},\dots, C_{n}$. Fix a
chord $c$ of it. By {\em surgery along $c$} we mean the following
operation. We delete small neighbourhoods of endpoints of $c$ and
connect the obtained endpoints by segments in the following way.
There are $3$ ways of pairing the four points. One of them
corresponds to the disconnection we have performing. We choose
another way as follows. If $c$ is positive then we connect the
endpoints according to the orientation of circles, and if $c$ is
negative, we connect the endpoints in the way opposite manner, see
Fig. \ref{surgery}

\begin{figure}
\centering\includegraphics[width=200pt]{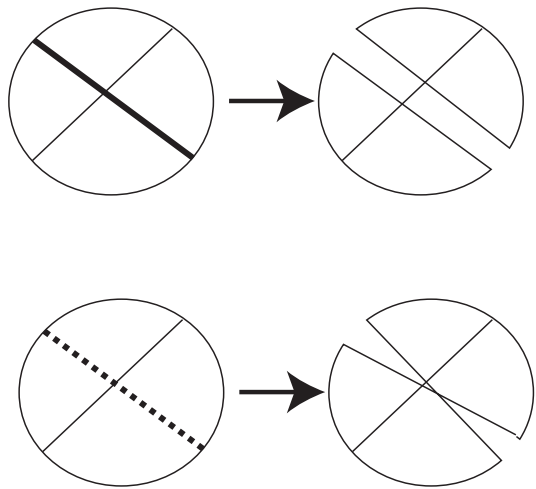}
\caption{Surgeries along positive and negative chords}
\label{surgery}
\end{figure}

Then we get a collection of circles, not necessarily oriented. If we
choose a collection of chords $c_{1},\dots, c_{k}$ of $C$, then the
surgery along these chords means the consequence of surgeries
performed along all chords $c_{i}$; in each case we look at the
orientations of the initial diagram $C$.

Assume the circle represents a boundary component of an annulus. By
adding a band to an annulus circle transforms its boundary component
according to a surgery along the chord corresponding to the band,
see Fig. \ref{surgeryband}.

\begin{figure}
\centering\includegraphics[width=200pt]{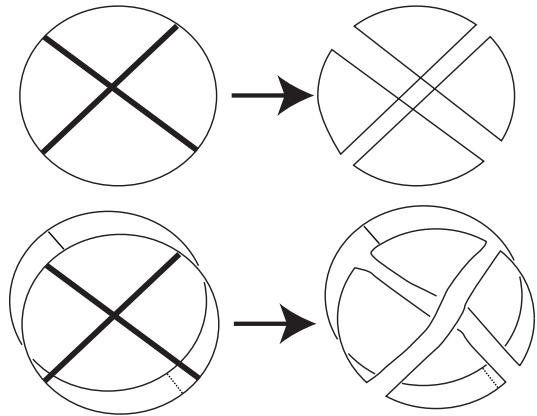} \caption{A
surgery and its band representation} \label{surgeryband}
\end{figure}

Thus, the number of circles in the state corresponding to the chords
$d_{1},\cdots, d_{n}$ is precisely the number of components of the
manifold obtained from the initial circle by surgery along these
chords.

Now, for a framed graph $G$ on $k$ enumerated vertices, introduce
the intersection matrix of $G$ to be $k\times k$ matrix over ${\bf
Z}_{2}$ whose rows and columns correspond to vertices of $G$ such
that $M_{ij}=M_{ji}=1$ for $i\neq j$ iff the vertices $i,j$ are
connected by an edge and $M_{ii}=1$ iff $i$-th vertex is framed
negatively.

Surprisingly, this number can be counted out of the intersection
graph even when the corresponding chord diagram does not exist, due
to the following

\begin{thm}[Soboleva,\cite{Soboleva}]
For a chord diagram $D$ with an intersection graph $G$ the number of
components of the manifold obtained from $D$ after a surgery along
chords $1,\dots, k$ is one plus the corank of $M(D)$.
\end{thm}

Now, we just define for a framed graph $\Gamma$ the Kauffman bracket
as

\begin{equation}
\sum_{G'\subset G}a^{2|G'|-n}(-a^{2}-a^{-2})^{corank M_{\Gamma'}}
 \label{Kauffmanbracketforgraphs}
\end{equation}

Soboleva's theorem allows to reformulate Problem 1 as Problem 2.
Indeed, let $\Gamma$ be a framed $4$-graph. We are looking for an
embedding of $\Gamma$ into a surface $S$ of minimal (maximal) genus
with a checkerboard face colouring. Choose a rotating circuit of
$\Gamma$ and a corresponding framed chord diagram $C(\Gamma)$.

Assume $\Gamma$ is embedded in a certain surface $S$. Then
$C(\Gamma)$ yields a mapping $S^{1}\to S$ which is an embedding
outside pre-images of vertices of $\Gamma$. This map can be slightly
smoothed to give an embedding as shown in Fig. \ref{blackwhite}.

\begin{figure}
\centering\includegraphics[width=200pt]{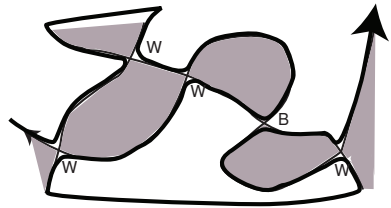} \caption{A
circuit splits chords into two sets} \label{blackwhite}
\end{figure}

Obviously, this circle $S^{1}\in S$ is {\em separating}: it divides
the surface $S$ into the ``white part'' and the ``black part''. We
can draw all chords of $C(\Gamma)$ as small edges on $S$ lying in
neighbourhoods of vertices of $\Gamma$. Thus, all chords of the
chord diagram $C(\Gamma)$ are naturally split into two families:
those connecting white regions and those connecting black regions.

Vice versa, any splitting of chords of $C(\Gamma)$ into two families
(black and white) gives rise to a checkerboard colourable embedding
of $\Gamma$ into a certain surface $S$. Indeed, consider an annulus
$S^{1}\times I$ and let $S^{1}$ be the medial circle of this
annulus. Now, we attach bands to two sides of the annulus according
to the splitting; a band is overtwisted iff the corresponding chord
is negative. This leads to a manifold $M$ with boundary; this
boundary naturally splits into two parts corresponding to the
boundary components of the annulus. Gluing the boundary components
of $M$, we get the desired surface $S$ without boundary, see Fig.
\ref{attachbands}.

\begin{figure}
\centering\includegraphics[width=400pt]{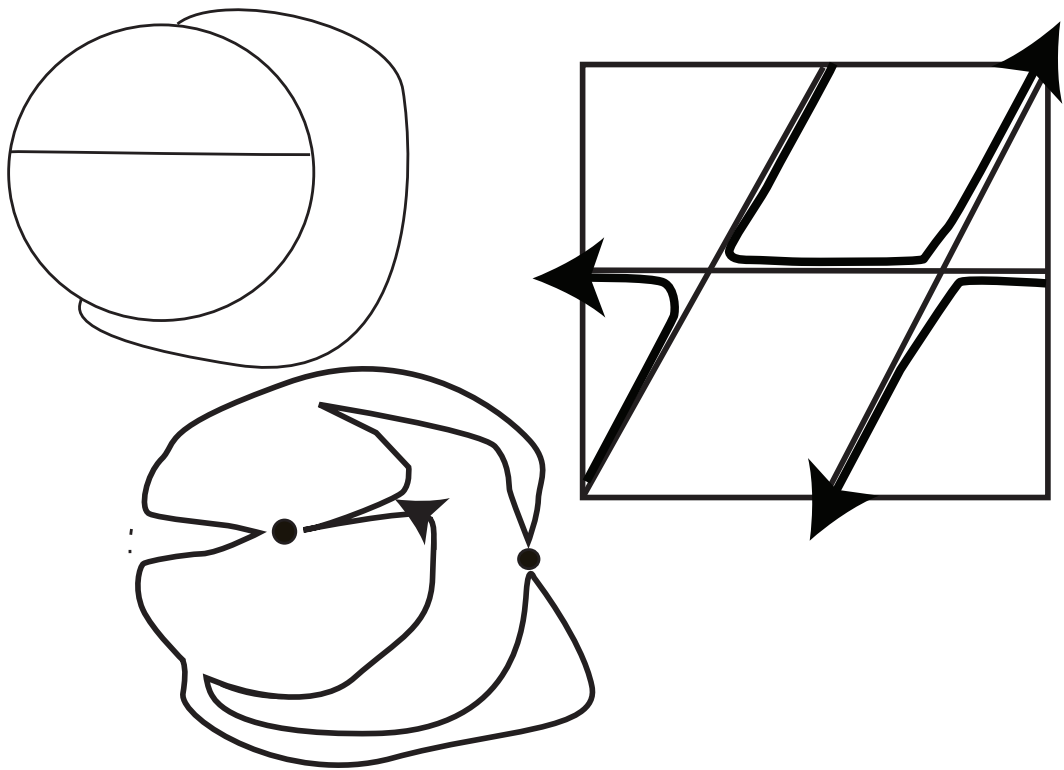}
\caption{Restoring the atom from a chord diagram with two sets of
chords} \label{attachbands}
\end{figure}

Thus, the question of estimating the genus (Euler characteristic) of
$S$ is equivalent to the question of maximising (minimising) the
boundary components of $S$. By definition, this is nothing but
counting the number of components of the two $1$-manifolds obtained
from the sphere by a surgery along the set of chords. By Soboleva's
formula, we have two subsets of chords $I$ and $J$, and we should
take two coranks of the adjacency matrices $M_{I}$ and $M_{J}$.

Thus, we have to find a way of splitting the chords in order to
maximise (minimise) the sum of ranks of the two matrices $rank
M_{I}+rank M_{J}$.

\begin{re}
Note that this solution does not depend on a particular choice of a
rotating circuit, depending only on the initial framed graph.
\end{re}

\begin{st}
Given a framed $4$-graph $\Gamma$ and the chord diagram $C(\Gamma)$
corresponding to some rotating circuit of $\Gamma$. Then if all
chords of $C(\Gamma)$ are positive then {\em all} checkerboard
colourable embeddings of $\Gamma$ yield orientable surface. If at
least one chord of $\Gamma$ is negative then all such surfaces are
non-orientable. \label{orientability}
\end{st}

\begin{re}
Note that the statement above means, in particular, that if {\em for
some} circuit the chord diagram contains a negative chord, then so
are all diagrams corresponding to all rotating circuits for the same
graph.
\end{re}

\subsection{The source-target condition}

The above condition can be reformulated intrinsically in the terms
of graph.

\begin{dfn}
We say that a four-valent framed  graph satisfies the {\em
source-target condition} if each edge of it can be endowed with an
orientation in such a way that for each vertices some two opposite
edges are emanating, and the remaining two edges are incoming.
\end{dfn}

Obviously, for a given graph there exists at most one source-target
structure (up to overall orientation reversal of all edges).
Moreover, if such a structure exists, then it agrees with {\em any}
rotating circuit. Namely, starting with a rotating circuit one may
try to orient its edges consequently in order to get a source-target
structure of the whole framed graph. The only obstruction one gets
in this direction corresponds to negative chords.

From the above, we get the following
\begin{thm}
A graph admits a source-target structure if and only if it is
checkerboard-embeddable into an orientable surface; in other words,
a source-target structure means that for any rotating circuit all
chords are positive.
\end{thm}

\subsection{The planar case: Vassiliev's conjecture}

For a $4$-graph $\Gamma$ to be embedded in ${\bf R}^{2}$ (or
$S^{2}$), take a chord diagram $C(\Gamma)$ corresponding to some
rotating circuit of $\Gamma$, and consider the adjacency matrix
$M_{C({\Gamma})}$. A simple calculation shows that the corresponding
sum of ranks should be {\em the minimal possible}, i.e., equal to
zero. That means that all chords of $\Gamma$ are positive (otherwise
we would have diagonal non-zero entries giving rank at least $1$).
Moreover, the chords should constitute two families of
non-intersecting chords (each family forming a submatrix of rank
$0$). This means that the corresponding intersection graph is
bipartite or the diagram is a $d$-diagram.

Now, one can check that for a diagram with a negative chord $c$
there is a Vassiliev obstruction consisting of two circuits having
precisely one intersection point at the vertex of $\Gamma$
corresponding to $c$. Besides, for a chord diagram which is not a
$d$-diagram, one can explicitly construct a Vassiliev obstruct. This
leads to a proof of Vassiliev's conjecture. For more details see
\cite{VassConj}.

Note that the above considerations lead to a fast (quadratic on the
number of chords) algorithm of planarity recognition: one takes any
circuit, checks whether all chords are positive, and then checks
that a diagram is a $d$-diagram. The latter consists of possible
splitting of all chords into two disjoint sets, which is unique for
chord diagrams with connected intersection graphs.

\subsection{The case of ${\bf R}P^{2}$}

Here we shall use the adjacency matrix. According to Statement
\ref{orientability}, the adjacency matrix should have at least one
non-zero diagonal element. Without loss of generality, assume it is
$(1,1)$. Since we are looking for a splitting of $\{1,\dots, n\}$ in
order to get rank $1$, all elements entries $m_{jj}=1$ should belong
to the same set. Without loss of generality, assume $a_{11}=\dots =
m_{kk}=1, m_{k+1,k+1}=\dots m_{nn}=0$. Now, merge some subset
$\{k+1,\dots, n\}$ with $\{1,\dots, n\}$ and leave the remaining
part as it is in order to get the total rank $1$. The remaining part
should thus have rank $0$, while the former should not increase rank
$1$ formed by the first $k$ entries of the matrix. This can be done
by the procedure similar to finding $d$-diagrams. The generic
diagram corresponding to ${\bf R}P^{2}$ looks as follows (see Fig.
\ref{asfollows}): there is a family of dashed chords (all
intersecting each other) and two families of pairwise disjoint
chords; chords belonging to one family do not intersect dashed
chords.

\begin{figure}
\centering\includegraphics[width=200pt]{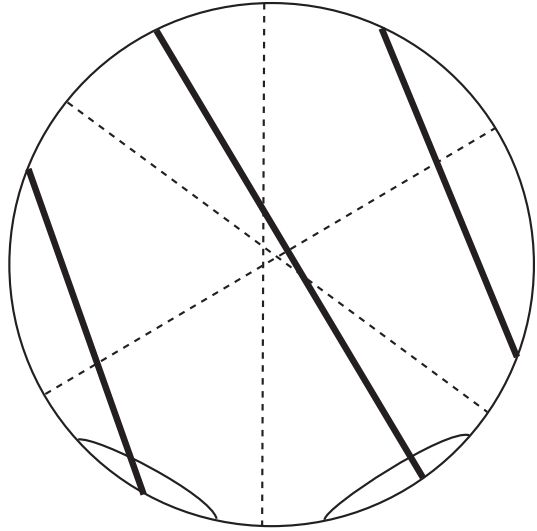} \caption{A
generic framed chord diagram corresponding to a graph in ${\bf
R}P^{3}$} \label{asfollows}
\end{figure}

In Fig. \ref{asfollows} solid chords from another family are
represented by thicker lines than chords belonging to the family
containing all dashed chords.

Obviously, the algorithm described in the present section has
quadratic complexity.

\subsection{The case of the Klein bottle}

The main idea of detecting the Klein bottle embeddability is the
following. While seeking the minimal rank $2$ there might be two
possibilities: either $2=1+1$ (which corresponds to the Klein bottle
represented as a connected sum of two projective planes or $0+0$;
the first case is easier, and it turns out that the general case can
be reduced to it).

\begin{lm}
For every four-valent framed graph embeddable in ${\bf KL}^{2}$
there exists a rotating circuit dividing ${\bf KL}^{2}$ into two
copies of ${\bf R}P^{2}$.
\end{lm}

\begin{proof}
Starting with a rotating circuit bounding a disc, one can get a
desired circuit by performing surgery along a dashed chord, see Fig.
\ref{seefig2}.

\begin{figure}
\centering\includegraphics[width=250pt]{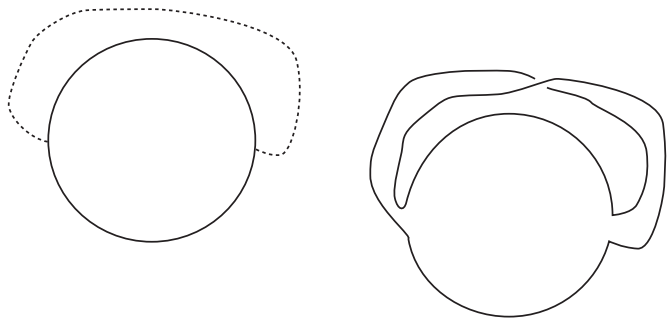} \caption{A
surgery along a dashed chord generates a M\"obius band}
\label{seefig2}
\end{figure}

\end{proof}

Then, the procedure is as follows: we take a dashed chord, perform a
surgery and look whether the intersection matrix corresponding to
the obtained chord diagram is splittable into two families giving
the desired decomposition.

The desired decomposition should be such that inside each family any
two dashed chords intersect, and any non-dashed chord is disjoint
from any other chord. Thus, the procedure of finding two families is
exactly the same as in the case of $d$-diagrams, except for the case
of incidence of dashed chords.

\subsection{The chord diagram algebra and the graph algebra}

Chord diagrams play a crucial role in the study of finite-type
(Vassiliev) invariants of knots \cite{Vass,BN}. Roughly speaking,
for every positive integer $n$ there is a class of invariants of
degree $n$ whose ``leading term'' (called symbol or $n$-th
derivative) is a function on chord diagrams satisfying a certain
relation. Invariants of the same order having the same leading term
differ by an invariant of a strictly smaller order (like polynomials
of degree $n$ whose $n$-th derivative coincide).

There are two versions of the chord diagram algebra: for usual knots
(with two sorts of relations, the four-term (see ahead) and the
one-term) and for framed knots (with only the four-term relation).

We shall restrict ourselves for the case of only four-term relation
which is defined as follows (see Fig. \ref{4T}): given four diagrams
on $n$ chords each whose $n-2$ chords coincide (they are not
depicted in Fig. \ref{4T} and have endpoints in punctured areas) and
the disposition of the remaining two chords, $\alpha$ and $\beta$ is
as shown in Fig. \ref{4T}.

\begin{figure}
\centering\includegraphics[width=400pt]{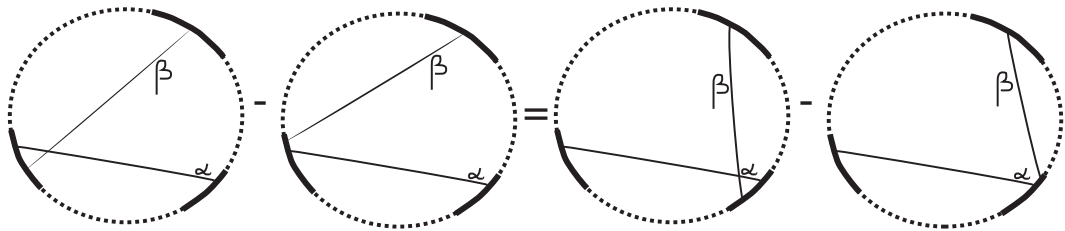} \caption{The
four-term relation} \label{4T}
\end{figure}

\begin{re}
There exists a standard ``deframing'' procedure which associates
with each weight system satisfying only the $4T$-relation a weight
system satisfying both the $4T$ and the $1T$-relation. The latter
means that the diagram containing a solitary chord is equal to zero.
\end{re}

We define the linear space ${\cal A}^{c}_{n}$ (over ${\bf Q}$) as
the quotient space of all chord diagrams on $n$ chords modulo the
four-term relation.

It turns out that the space ${\cal A}^{c}=\oplus_{i=0}^{\infty}{\cal
A}^{c}_{n}$ has an algebra (and even bialgebra and Hopf algebra)
structure. Namely, for multiplication of two chord diagrams we break
each of them at some point and glue together with respect to the
orientation. This operation is well-defined modulo $4T$-relation
\cite{MyBook}.

The comultiplication operation for a chord diagram $C$ is defined as
$\Delta(C)=\sum_{I\sqcup J}C_{I}\otimes C_{J}$, where the sum is
taken over all possible ways to split the total set of chords into
two subsets $I$ and $J$, and we take the subdiagrams $C_{I}$ and
$C_{J}$ formed by these sets.

Having the intersection graph mapping, one can define the analogous
operations on graphs. The multiplication operation is defined even
easier than in the case of chord diagram: there is no need to prove
its well-definiteness. On the other hand, one can define (following
Lando, \cite{CDL}) the $4T$-relation. It is defined as follows.

To make the definition of the graph algebra precise, we have to
describe the correspondence between the $4$ terms $A,B,C,D$ in the
graph-theoretic language. Since vertices of the graph correspond to
chords, the diagrams $A$ and $B$ differ just by one chord: in $A$,
the vertices $\alpha$ and $\beta$ are connected by an edge, whence
for $B$ they are not. The same for $C$ and $D$. It remains only to
explain how to get $C$ from $\alpha$. In the chord diagram we moved
one end of the chord $\beta$ from one end of the chord $\alpha$ to
the other while passing from $A$ to $C$, see Fig. \ref{4T}. This
means that the chord $\beta$ changes its incidence with all chords
incident to $\alpha$ and does not change its incidence with chords
which are not incident to $\alpha$.

\begin{re}
Note that the surgery over $A$ and the surgery over $C$ lead to the
same number of circles; the same is true about $B$ and $D$.
\end{re}

It turns out that many nice functions defined on graphs (e.g., the
Tutte polynomial) satisfy the $4T$-relation. We shall touch on such
functions while defining the generating function for counting genera
of surfaces spanning a given graph.

It turns out that the chord diagram algebra has its meaningful
``framed analogue''. It is connected to so-called finite-type
invariants of plane curves, see \cite{Lando} for details. The
definition goes as follows.

Consider the set of all framed chord diagrams on $n$ chords. Now, we
set $A^{cf}_{n}$ to be the ${\bf Q}$-linear space generated by all
such chord diagrams subject to the generalised $4T$-relations
depicted in Fig. \ref{gen4T}.

\begin{figure}
\centering\includegraphics[width=300pt]{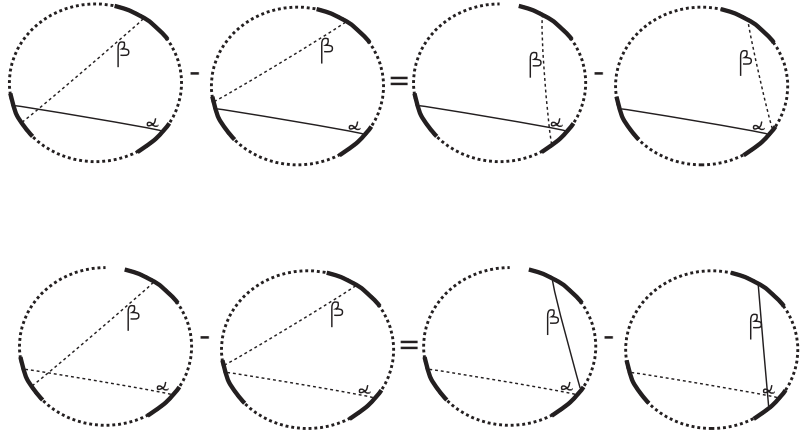}
\caption{Two types of generalized 4-term relations} \label{gen4T}
\end{figure}

The general rule for the generalized four-term relation is as
follows: we have $n-2$ fixed chords and two chords, $\alpha$ and
$\beta$. If $\alpha$ is positive then in all chord diagrams
$A,B,C,D$, the chord $\beta$ is of the same sign (in all cases
positive or in all cases negative), and the relation looks just as
the usual four-term relation: $A-B=C-D$. If $\alpha$ is negative,
then while moving from $A,B$ to $C,D$ the chord $\beta$ changes its
sign; moreover, the RHS changes the overall sign: $A-B=D-C$, that
means that if in the LHS we take the chord diagram with intersecting
$\alpha,\beta$ with plus then in the RHS we take the diagram with
intersecting $\alpha,\beta$ with minus.

One can easily check that in any special case of the generalized
$4T$-relation, the surgery along $A$ gives the same number of
circles as the surgery along the diagram with plus in the right hand
side ($C$ or $D$) whence the surgery along $B$ gives the same number
of circles as the surgery along the diagram with minus from the RHS.

To the best of the author's knowledge, the connected sum operation
on the framed chord diagram algebra is not proved to be
well-defined. Of course, the coalgebraic operation is well-defined.

Analogously, one defines the bialgebra of framed graphs (at the
level of graphs, there is no problem to define the product, we omit
the exact definition leaving it for the reader as an exercise).

\section{The generating function for the embedding genera}

 \subsection{Weight systems associated with Lie algebras: a brief
 review}

There is a natural way of associating a number with a given chord
diagram and a given representation of a (semisimple) Lie algebra due
to Bar-Natan, \cite{BN}. It turns out that the corresponding mapping
naturally extends (for a fixed representation $R$ of a fixed Lie
algebra $G$) to the mapping from the algebra of chord diagram ${\cal
A}^{c}$ to ${\bf Q}$ because of the similarity of the $4T$-relation
and the Jacobi identity.

We shall deal only with the adjoint representation of Lie algebras;
for a Lie algebra $G$ we denote the corresponding mapping from
${\cal A}^{c}$ to ${\bf Q}$ by $W_G$.

The construction goes as follows. Every chord diagram is a cubic
graph immersed in the plane with prefixed rotation direction at each
vertex (when drawing chord diagrams on the plane we assume this
rotation to be counterclockwise). We shall enlarge the construction
for arbitrary cubic graphs with rotation. Namely, we take the
structural tensor $L_{ijk}$ of the Lie algebra ${\bf G}$ with all
indices shifted down by using the Cartan-Killing metric. Obviously,
$L_{ijk}=L_{jki}=L_{kij}$. Now, we can associate with each trivalent
vertex the tensor $L_{ijk}$ with indices corresponding to edges and
going counterclockwise $i,j,k$. Then we contract all tensors along
edges (of course, by using Cartan-Killing metric tensor) and get an
integer.

We specify ourselves for the case of the Lie algebra $sl(n)$ and its
adjoint representation.

In \cite{Diplomarbeit}, see also \cite{MitRutwig}, we proved the
following

\begin{thm}
For a given chord diagram $D$ on $n$ chords, $W_{sl_{n}}(D)$ is a
polynomial in $n$; its degree does not exceed $k+2$; moreover, it is
equal to $k+2$ only in the case when $D$ is a $d$-diagram.
\end{thm}

As an immediate consequence from this theorem we see that each basis
of the chord diagram algebra consisting of chord diagrams contains
at least one $d$-diagram.

On the other hand, it underlines the special role of $d$-diagrams
amongst all chord diagrams from two points of view: as those
(corresponding to graphs) embeddable in $S^{2}$ and as those having
the highest possible degree of the leading term.

It turns out that this is not an incident: degrees of the
$W_{sl_{n}}$ polynomial are closely connected to possible embeddings
of the $4$-valent graph into surfaces. We shall touch on this
subject in later sections.

\subsection{Checkerboard colourable embeddings}

As we have seen, in order to mimimise (maximise) the genus of the
surface the graph can be embedded into, we have to maximise
(minimise) the number of circles obtained as a result of surgeries
along two subsets $I,J$ of chords of the initial chord diagram such
that $I\sqcup J$ forms the complete set of chords.

Before we have reformulated this problem in terms of ranks of
incidence matrices. A new formulation comes with the generating
function.

Let $\Gamma$ be a framed $4$-valent graph on $k$ vertices, and let
$D$ be a chord diagram corresponding to some circuit of $\Gamma$.
Consider the following function

\begin{equation}f(C):=f(\Gamma)= \sum_{atoms}x^{k+2-g},\label{xiyouji}\end{equation}
where the sum is taken over all atoms with $A$-structure taken from
$\Gamma$ and $g$ is the genus of the atom.

In view of Soboleva's theorem, $f(C)$ depends merely on the
intersection graph of $\Gamma$.

Consider the restriction of $f(C)$ to chord diagrams with only
positive chords (i.e., to graphs corresponding to orientable atoms).

\begin{thm}
$f(C)$ is a well-defined function on the algebra of chord diagrams,
i.e., it satisfies the $4T$-relation.

Analogously, $f(C)$ is a function on the graph algebra.

Moreover, $f(C)$ is multiplicative with respect to the
multiplication operations in these algebras.\label{turntoproof}
\end{thm}

Both chord diagram algebra and graph algebra have a commutative and
co-commutative Hopf-algebra structure, see, e.g.,
\cite{BN,CDL,MyBook}. By Milnor-Moore theorem, \cite{MilnorMoore},
each such algebra is isomorphic to the polynomial algebra of its
primitive elements. Thus, in order to calculate $f(C)$ for a given
chord diagram, one can use the algebraic structure of the Hopf
algebra of chord diagrams (or graphs).

Now, we turn to the proof of Theorem \ref{turntoproof}. Consider a
quadruple of chord diagrams $A,B,C,D$ on $n$ chords each forming a
$4T$-relation $A-B=C-D$ as shown in Fig. \ref{4T}. We can naturally
identify chords from $A,B,C,D$: there are $n-2$ chords in common,
one ``fixed chord'' (denoted by $\alpha$ in Fig. \ref{4T}) and one
``moving'' chord (denoted by $\beta$).

Consider summands for $f(A),f(B),f(C),f(D)$ coming from the
definition (\ref{xiyouji}). For those summands where $\alpha$ and
$\beta$ belong to the same subset of chords (say, $I$), the genus of
surface corresponding to $A$ is equal to the genus corresponding to
$C$, and the genus corresponding to $B$ is equal to the genus
corresponding to $D$ (this follows from a straightforward
calculation of the number of circles). Thus, these terms give the
same contribution to (\ref{xiyouji}).

For those summands where $\alpha$ and $\beta$ belong to different
subsets $I$ and $J$, the corresponding subdiagrams coincide:
$A_{I}=B_{I},A_{J}=B_{J},C_{I}=D_{I},C_{J}=D_{J}$ because when we
move $\alpha$ and $\beta$ to different subdiagrams, it does not
matter whether they intersect or not.

The proof for the graph algebra is analogous.

Arguing as above, one can prove the following
\begin{thm}
The restriction of function $f$ to ${\cal A}^{c}$ satisfies the
generalised $4T$-relation.\label{mth1}
\end{thm}

\begin{crl}
If for a framed graph $\Gamma$ satisfying source-target condition on
$k$ vertices and the corresponding chord diagram $C$ we have $deg
f(C)=k$ then $\Gamma$ is checkerboard-embeddable into torus.
\end{crl}

\begin{proof}
Indeed, the maximal possible degree $k+2$ corresponds only to planar
embedding; the orienability of the surface is guaranteed by the
source-target condition, and the degree $k$ corresponds to genus
$1$.
\end{proof}

It is important to know what sort of weight system we obtain from
$f$. It turns out that this weight system is closely connected to
$W_{sl_n}$; roughly speaking, $W_{sl_n}$ can be represented as a sum
of $2^{2k}$ summands (for $k$ chords); some $k$ of them give exactly
the function $f$. Moreover, the $W_{sl_n}$-polynomial itself gives a
``generating function'' for {\em some more general embeddings} (see
ahead), however, this generating function has signs $\pm$, so the
embeddings are counted with pluses and minuses, which means that not
the whole information can be restored from $W_{sl_n}$.

To clarify the situation, we shall need some more information about
calculating $W_{sl_n}$ (see \cite{Diplomarbeit} and
\cite{MitRutwig}). We will in fact work in $gl_{n}$; the result of
final contraction will be the same as that for $sl_{n}$.

Given a chord diagram $D$, fix an arc of it and break this diagram
along the arc. Then $W_{sl_n}(D)=Tr(x\to [\dots[,x]\dots]$ where by
$[\dots[,]\dots]$ we mean the result of consequent commutators of
$x$ with elements of the Lie algebra, where for each chord we take
$\alpha$ on one end of the chord and the dual element $\alpha^{*}$
on the other end of the chord and sum up when $\alpha$ runs the
basis of the Lie algebra. Let us be more specific. Consider the
diagram shown in Fig. \ref{NyFig}.

\begin{figure}
\centering\includegraphics[width=300pt]{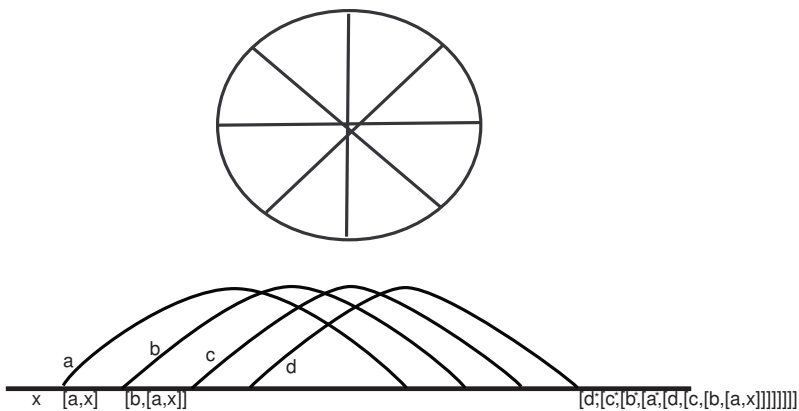} \caption{The
contraction operator} \label{NyFig}
\end{figure}

The ``long'' commutator can be rewritten according to $[p,q]=pq-qp$.
Thus we get $2^{2k}$ terms of the following form $(-1)^{l}p_{1}\dots
p_{l}x q_{2k-l}\dots q_{1}$, where $p_{1},\dots p_{l}$ are variables
corresponding to some chord ends in the usual order, and
$q_{2k-l},\dots, q_{1}$ are the remaining chord ends in the reversed
order. For instance, for a diagram shown in Fig. \ref{NyFig}, we get
summands like $(-1)d^{*}baxcda^{*}b^{*}c^{*}$.

Now, two simple $gl(n)$-contraction formulae come into play:

\begin{equation}
Tr(A\alpha B\alpha^{*})=Tr(A)Tr(B);\
Tr(A\alpha)Tr(B\alpha^{*})=Tr(AB). \label{okri}
\end{equation}

Here the sum is taken over $\alpha$ running over a basis of
$gl_{n}$, whence $\alpha^{*}$ runs the dual basis; $A$ and $B$ may
be arbitrary $n\times n$ matrices.

With these formulae, we may contract any formula of the sort
described above.

The formulae (\ref{okri}) have the following geometric meaning: if
$A$ and $B$ represent collections of some chord ends (lying on one
or two circles of a chord diagram) then the contraction transforms
one circle into two or two circles into one.

{\bf This means precisely that the contraction rules in $sl_{n}$
correspond to surgery operations.}

In the very beginning, we may contract $x$ with itself: this will
lead to a subdivision of the circle into two circles according to
the way of choosing variables on the left hand side with respect to
$x$ and variables on the right hand side with respect to $x$.

Schematically, it means that we have to collect $2^{2k}$ terms
corresponding to different splittings of chord ends into two
circles. For instance, $(-1)d^{*}baxcda^{*}b^{*}c^{*}$ corresponds
to the diagram shown in Fig. \ref{dgrm}.

\begin{figure}
\centering\includegraphics[width=200pt]{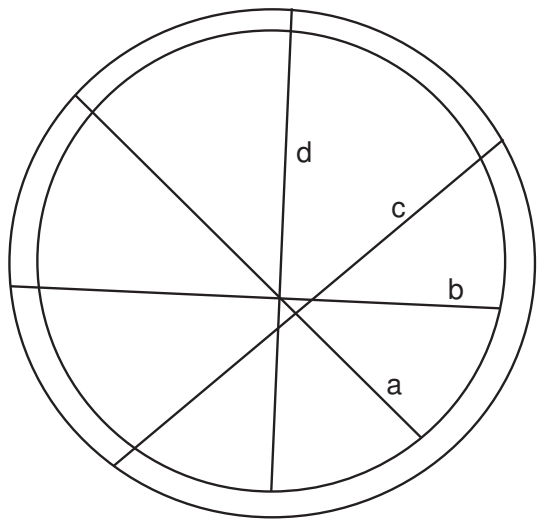} \caption{A chord
diagram on 2 circles corresponding to
$(-1)d^{*}baxcda^{*}b^{*}c^{*}$} \label{dgrm}
\end{figure}

Call a summand {\em good} if for each chord both ends belong to the
same subset. Such summands contribute with sign $+$. Now, it follows
from definition that the contribution of good summands gives exactly
the function $f$.

Now, we the geometric meaning of the function $f$ itself.

\subsection{Embeddings with orienting ${\bf Z}_{2}$-homology class}

For generic summands (not necessarily good) we deal with arbitrary
ways of splittings of chord ends into two sets. This leads to an
arbitrary way of attaching bands to the annulus, and, finally, gives
the generating function for arbitrary embeddings of our framed graph
provided that the ${\bf Z}_{2}$-homology represented by this graph
corresponds to an orienting cycle.

Let us be more specific. First of all, our weight systems
corresponding to Lie algebras are defined only for the case when all
chords of the chord diagram are positive. On the other hand, the
generating function can be written down for an arbitrary graph. As
we shall see, all generating functions will satisfy the generalized
$4T$-relation and give a certain generalized weight system in the
sense of Lando.

Analogously to $f$, we define the function ${\tilde f}$ as follows:
\begin{equation}{\tilde f}(C):={\tilde f}(\Gamma)= \sum_{{\bf Z}_{2}-\mbox{orient.emb}}x^{k+2-g},\label{xiyouji'}\end{equation}
where the sum is taken over ${\bf Z}_{2}$-orientable embeddings.

Analogously to Theorem \ref{mth1}, one can prove \begin{thm} The
restriction of ${\tilde f}$ to chord diagrams satisfies the
$4T$-relation.
\end{thm}

This immediately yields the following

\begin{crl}
If ${\tilde f}(C)$ has maximal degree $k+2$ then the corresponding
framed $4$-valent graph is checkerboard-embeddable into the torus.
\end{crl}

Unfortunately, the inverse is not true: the graph corresponding to
the chord diagram consisting of three pairwise-linked chords is
embeddable into torus, though the $sl(n)$-function on this chord
diagram gives zero.

 Now, let us try to understand the geometric
meaning of ${\tilde f}$. For a chord diagram on $k$ chords, the
function $f$ can be represented by $2^{2k}$ summands which clearly
correspond to some contraction rules in $sl_{n}$. These are exactly
the summands corresponding to those contractions where we place both
ends of each chord on the same circle.

It would be very interesting to understand the nature of contraction
along chords connecting points on different circles, see Fig.
(\ref{dgrm}).

First of all, in the expansion for a commutator, we may get a minus
sign, which corresponds to a surgery along a chord diagram with two
ends on different circles.

Now assume we perform count the $sl(n)$ weight system for a chord
diagram. As we see, after expanding the commutators, the expressions
for two circles go in the opposite order. Thus, in order to restore
the real picture of embedding genera, one should perform overtwisted
surgeries along chords with endpoints on different circles.

Thus, $W_{sl_n}$-weight system estimates the genera of the surfaces
the graph can be embedded to, but:

\begin{enumerate}

\item It counts embeddings with signs, thus, for some embedding it
does not give a real estimate for the genus. For instance, for the
chord diagram with $3$ pairwise intersecting chords the value of the
$W_{sl_n}$-function is zero.

\item The contraction corresponding to chords with ends on different
circles count embeddings of the same graph, but with another
framing.

\end{enumerate}

So, the geometric meaning of ${\tilde f}$ is not yet completely
understood.

\subsection{The general case}

We have described how to write the generating function of all
embeddings for a given graph into surfaces such that the ${\bf
Z}_{2}$-homology class of the graph is orienting.

Namely, let $\Gamma$ be a framed graph and let $C$ be its rotating
circuit. If we want to embed $\Gamma$ into some surface in such a
way that $\Gamma$ represents a non-orienting ${\bf Z}_{2}$-cycle
then this condition depends on $\Gamma$ itself and does not depend
on $C$. Thus, we have to start with embedding the neighbourhood of
$C$ as the medial circle of the M\"obius band. Then we attach bands
to this M\"obius band in the usual way (according to the
$A$-structure of $\Gamma$) and paste the boundary components by
discs.

At the level of chord diagrams this means the following. Instead of
taking two copies of the chord diagram circle, we take a $2$-fold
covering of the circle by a circle and distribute the $2n$ chord
endpoints in one of two possible positions each, and then perform
contraction, see Fig. \ref{contractnonorientable}.

\begin{figure}
\centering\includegraphics[width=300pt]{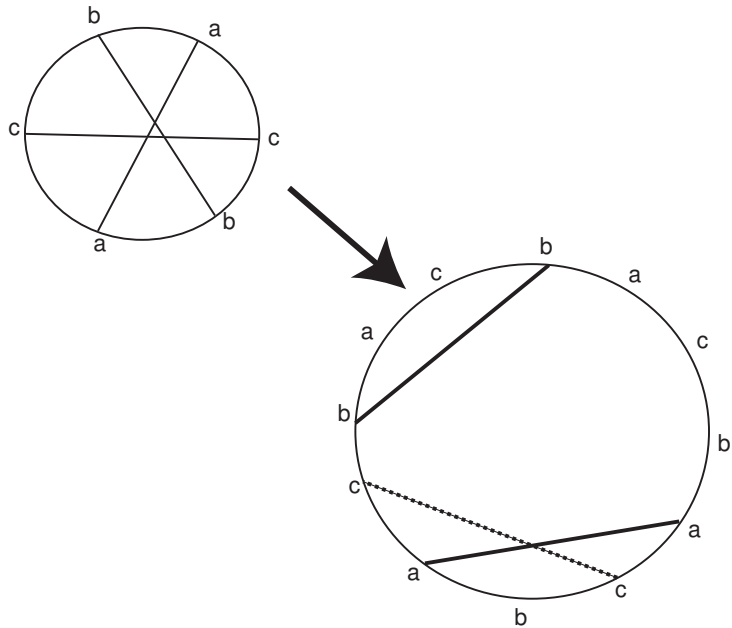}

\caption{$2$-fold covering of a chord diagram}
\label{contractnonorientable}
\end{figure}

\section{Unsolved problems}

The method of matrix ranks gives an explicit polynomial solution
only in a very limited number of cases. It is known that for every
surface of rank $g$ there is a solution to Problem 1 which is
polynomial in the number of chords. It would be very interesting to
get such algorithms via matrices.

Both bialgebra of chord diagrams and coalgebra of framed diagrams
are well known. At the level of chord diagrams their structures are
well studied.

However, weight system approach is applicable only to chord diagrams
with solid chords (which correspond to chord diagrams satisfying the
source-target condition). It is not yet known how to apply any
similar techniques for framed chord diagrams (with ``dashed''
chords). Possibly, one should treat positive chords by means of
$sl(n)$-tensors, and negative chords by means of $so(n)$ or maybe
$sp(n)$-tensors (cf. \cite{BN} and \cite{MitRutwig}).

It is well known (see \cite{LandoZvonkine}) that many enumerative
problems in graph theory can be solved by using Gaussian integrals.
However, these problems usually count generating functions for
genera coming from {\em all possible} gluings, say, of a polygon. In
our problem, we have to fix a graph (or a chord diagram), and
consider the generating function for genera of surface this graph
can be embedded into. Possibly, this can be done by means of a Gauss
integral for all {\em admissible} gluings of crosses at vertices of
the diagram.

A four-valent graph with $A$-structure can be represented as a
shadow of a virtual link. Problem 1 is devoted to finding the
minimal atom genus for a link with this shadow and some classical
crossing setup. Fixing such a link with one white cell leads to a
chord diagram and a rotating circuit. The Kauffman bracket of this
link is very similar to the generating function for the solution of
Problem 1: it has $2^{n}$ summands. It is known \cite{Mel} that the
Kauffman bracket (after some variable change) giving a series of
weight systems integrates to give the Kauffman $2$-variable
polynomial of the knot. It would be interesting to know which are
the knot invariants that can be obtained by integrating the
generating functions for Problem 1 and which knot-theoretic
properties they can detect.

\end{document}